\documentclass{article}
\usepackage{amssymb}
\usepackage{amsthm}
\usepackage{amsmath}
\usepackage{stmaryrd}
\usepackage{xcolor}

\newtheorem{dfn}{Definition}
\newtheorem{prop}{Proposition}
 
\newtheorem{theorem}{Theorem}

\newcommand{\M}{\mathcal{M}}
\newcommand{\FM}{\mathcal{F}^{\mathcal{M}}}

\newcommand{\stratum}{\mathcal{H}(\kappa)}

\newcommand{\proba}{\mathcal{P}(\M_1)}

\title{Isoperiodic dynamics in rank 1 affine invariant orbifolds}
\author{Florent Ygouf}

\begin{document}

\maketitle

\abstract{Let $\M$ be a rank 1 affine invariant orbifold in a stratum of the moduli space of flat surfaces. We show that the leaves of the $\M$-isoperiodic foliation are either all closed or all dense. In the second case, we establish ergodicity of the foliation with respect to the affine measure on $\M$.}

\section{Introduction}

\subsection{Context}

Any stratum $\stratum$ of the moduli space of translation surfaces is endowed with a holomorphic foliation $\mathcal{F}$ called the isoperiodic foliation, also known in the literature as the Kernel foliation, the absolute period foliation or the Rel foliation. The leaves of this foliation are locally described by modifiying the relative position of the singularities of a translation surface while keeping its (absolute) periods unchanged. Several papers have been dedicated to the dynamics of its leaves. McMullen has established the ergodicity of the foliation with respect to the Masur-Veech measure in the principal strata of genus 2 and 3, see \cite{mcmullen2014moduli}. Building upon this work, Calsamiglia, Deroin and Francaviglia have obtained in \cite{calsamiglia2015transfer} a classification of the closed saturated sets of the foliation in the principal stratum in any genus and the ergodicity with respect to the Masur-Veech measure. The ergodicity part in the principal strata has also been obtained independently by Hamenstädt in \cite{hamenstadt2018ergodicity} with different methods. Outside the principal strata, Hooper and Weiss gave an example of a dense leaf in $\mathcal{H}^{odd}(g-1,g-1)$ for abitrary $g \geq 3$ in \cite{hooper2015rel} and it is shown in \cite{ygouf2020criterion} that the leaves of non primitive Prym eigenforms in $\mathcal{H}(2,1,1)$ are dense.

\subsection{Statement of the results}

In this text we study the dynamics of a variation of the isoperiodic foliation in affine invariant orbifolds \footnote{In the literature, they are usually referred to as affine invariant manifolds but they are only suborbifolds, hence our choice of terminology.} of the moduli space. Those are suborbifolds that arise as $GL_2(\mathbb{R})$-orbit closures as has been established in the seminal work of McMullen in genus 2 and Eskin, Mirzakhani and Mohammadi in complete generality. See \cite{wright2014translation} for more details. If $\M$ is such an affine invariant orbifold, it is explained in \cite{ygouf2020criterion} how isoperiodic deformations inside $\M$ fit into a holomorphic foliation of $\M$. The following definition is a summary of section 2.2 of \cite{ygouf2020criterion}

\begin{dfn}
Let $\M$ be an affine invariant orbifold. The unit area sublocus $\M_1$ is endowed with a (immersed) foliation $\FM$ called the $\M$-isoperiodic foliation such that the leaf $\FM_x$ of any $x \in \M_1$ is a connected component of $\mathcal{F}_x \cap \M$.  
\end{dfn}

We emphasize that in \cite{ygouf2020criterion} the $\M$-isoperiodic foliation was defined as a foliation of the whole affine invariant orbifold. We adopt another convention here and consider its restriction to $\M_1$. This is licit as isoperiodic deformations preserve area. This is explained in section 2.2 of \cite{ygouf2020criterion}. The isoperiodic foliation $\mathcal{F}$ of the stratum $\mathcal{H}(\kappa)$ corresponds to the case where $\M = \stratum$. The rank 1 case is analyzed in \cite{ygouf2020criterion} and a criterion for density of isoperiodic leaves in that case is established. An application of that criterion provides a dichotomy for the topological behavior of the leaves in the case where $\M$ is a connected component of a Prym eigenform locus in genus 2 or 3: the leaves are all dense or all closed. In this text we generalize this result to abitrary rank 1 affine invariant orbifolds.

\begin{theorem}\label{main}
Let $\M$ be a rank 1 affine invariant manifold. Either all the leaves of the $\M$-isoperiodic foliation are closed or all the leaves are dense in $\M_1$. In the second case, the foliation is ergodic with respect to the affine measure. 
\end{theorem}

By affine measure we mean the only $SL_2(\mathbb{R})$-invariant probability measure supported on $\M_1$. See \cite{eskin2018invariant} for more details. The foliation is said to be ergodic with respect to the affine measure if any saturated Borel set is either null or conull for that measure. We recall that a set is said to be saturated if it contains all the leaves it intersects.

\subsection{Outline of the proof}

The proof of Theorem \ref{main} relies on the fact the action of $G=SL_2(\mathbb{R})$ on $\M_1$ permutes the leaves of the $\M$-isoperiodic foliation and acts transitively on the set of leaves. This set of leaves is thus isomorphic to a quotient of $G$ and carries a finite measure inherited from the affine measure. However, the measurable structure of that space is in general too pathological to be of any use and we introduce instead a measure-theoretic leaf space (that is incarnated by the space of probability measures on $\M_1$). We study the induced $G$-action and establish that stabilizers are now closed and have finite covolume in $G$. This is an avatar of the nicer measurable structure of our measure-theoretic leaf space. In vertue of Borel density theorem, stabilizers are either lattices of $G$ or $G$ itslef. This dichotomy underlies the dichotomy of Theorem \ref{main}.  

\subsection{Acknowledgements}

I am grateful to Erwan Lanneau for discussing with me this problem and other related questions as well as to Jean-Francois Quint for suggesting a shortcut in the proof of the main result. This text has greatly benefited from their insightful comments. This work has been partially supported by the LabEx PERSYVAL-Lab (ANR-11-LABX-0025-01) funded by the French program Investissement d’avenir.

\section{A measure theoretic leaf space}

Let $\M$ be a rank 1 affine invariant orbifold contained in $\stratum$ and denote by $m$ the corresponding affine measure. The space $\M_1$ is endowed with its Borel $\sigma$-algebra $\mathcal{A}$. We recall from \cite{ygouf2020criterion} that the action of $G$ on $\M_1$ plays out nicely with the $\M$-isoperiodic foliation. 

\begin{prop}\label{transitive}
The action of $G$ on $\M_1$ permutes the leaves of $\FM$ and acts transitively on the set of leaves. 
\end{prop}

Let $\proba$ be the set of probability measures on $\M_1$ endowed with the $\sigma$-algebra generated by the evaluation maps $\nu \mapsto \nu(A)$ where $A$ runs over the Borel sets. This is a standard Borel space. See section 17.E for more details. Let $\mathcal{A}_{sat}$ be the $\sigma$-algebra comprised of those Borel sets that are saturated by the foliation. We denote by $(m_x)_{x \in \M_1}$ the associated conditional measures. See section 5 of \cite{einsiedler2013ergodic} for more details on conditional measures. We define the following map: 

$$
\pi:
\begin{matrix}
\M_1 &\to& \proba \\
x &\mapsto& m_x
\end{matrix}
$$

The map $\pi$ is $\mathcal{A}_{sat}$-measurable by standard considerations on conditional measures. It is for instance item 1 of theorem $5.14$ \footnote{More precisely, item 1 of theorem 5.14 in \cite{einsiedler2013ergodic} states that only the restriction $\pi$ to the complement of a conull set of $\mathcal{A}_{sat}$ is $\mathcal{A}_{sat}$-measurable. We can then promote this map to a $\mathcal{A}_{sat}$ measurable map defined on $\M_1$ by assigning a constant value on the complement of the aforementionnend conull saturated set.} in \cite{einsiedler2013ergodic}. In particular, this means that $\pi$ is constant on the leaves of the $\M$-isoperiodic foliation as singletons in $\proba$ are measurable sets. This fact will be used several times later in the text. The following propostion justifies the relevance of the probability space $\proba$ in our study of the $\M$-isoperiodic foliation.  It essentially says that this space is the biggest quotient of $\M$ with a nice measurable structure that identifies points in the same leaves. 

\begin{prop}\label{factors}
For any standard Borel space $X$ together with a measurable map $f:\M_1 \to X$ that is constant on the leaves of the foliation, there is measurable map $\mathfrak{f}: \proba \to X$ such that $f = \mathfrak{f} \circ \pi$ $m$-almost everywhere. 
\end{prop}

\begin{proof}
We treat first the case where $X = [0,1]$. In that case $f$ is integrable with respect to any probability measure on $\M_1$ and we define $\mathfrak{f}(\nu) = \int_{\M_1} f d \nu$. By construction of the conditional measures (see item 1 of theorem 5.14 in \cite{einsiedler2013ergodic}) we have for $m$-almost every $x \in \M_1$: 

$$
\mathfrak{f} \circ \pi(x) = \int_{\M_1} f \ dm_x = E[f | \mathcal{A}_{sat}](x)
$$ 

Now, $f$ is $\mathcal{A}_{sat}$-measurable as it is constant on the leaves of the foliation and we deduce that $f$ is equal $m$-almost everywhere to its conditional expectation. Consequently, we proved that $f$ is equal to  $\mathfrak{f}\circ \pi$ $m$-almost everywhere.  

\medskip 

In the general case, since $X$ is a standard Borel space, there is measurable isomorphism $\phi$ from $X$ to a Borel subset of $[0,1]$. See for instance proposition A1 in the appendix of \cite{zimmer2013ergodic} for a proof of that claim. We can thus apply the previous case to the map $\phi \circ f$ to get a map $\mathfrak{h}$ and we set $\mathfrak{f} = \phi^{-1}\circ \mathfrak{h}$. 
\end{proof}

In the remainder of the text, we consider the induced action of $G$ on the space $\proba$ by pushforward. It is a measurable action and it preserves the measure $\mu = \pi_{\ast}m$. We denote by $m_G$ the Haar measure on $G$.   

\begin{prop}\label{equivariant}
There is a $G$-equivariant $\mathcal{A}_{sat}$-measurable map $\pi': \M_1 \to \proba$ that agrees with $\pi$ $m$-almost everywhere.
\end{prop} 

\begin{proof}
Since $m$ is $G$-invariant, unicity of conditional measures (see item 3 of theorem 5.14 in \cite{einsiedler2013ergodic}) implies that for any $g \in G$, the equality $gm_x = m_{gx}$ holds $m$-almost everywhere. Let $B$ be the set $\{x \in \M_1 \ | \ g \mapsto g^{-1}\pi(gx) \ \mathrm{is \ constant} \ m_G \mathrm{-almost \ everywhere}\}$. It is a consequence of Fubini's theorem that $B$ is measurable and conull. For $g,h \in G$ and any $x \in B$, we have: 

$$
g^{-1}\pi(ghx) = h\left((gh)^{-1}\pi(ghx) \right)
$$

\noindent This equality implies that $B$ is $G$-invariant. Since $G$ permutes the leaves of the $\M$-isoperiodic foliation and that $\pi$ is constant on those leaves, $B$ is also saturated and we deduce from proposition \ref{transitive} that $B$ is actually equal to $\M_1$. Now, for any $x \in \M_1$, we define $\pi'(x)$ to be the essential value of $g \mapsto g^{-1}\pi(gx)$. From the previous relation, it follows that $\pi'$ is $G$-equivariant. If $m_G'$ is any probability measure in the class of the Haar measure of $G$, we have the following: 

$$
\pi'(x) = \int_G g^{-1}\pi(gx) \ dm_G'
$$

This shows that $\pi'$ is measurable and constant on the leaves. Another application of Fubini's theorem shows that $\pi$ and $\pi'$ agree almost everywhere. 

\end{proof}

When a standard probability space $(\mathfrak{M},\mathfrak{m})$ is endowed with a measurable action of a group $\mathfrak{G}$, it is a standard fact that orbits are measurable sets. A proof of that claim can be found in corollary 2.1.20 of \cite{zimmer2013ergodic}. In that context, we say that the action of $\mathfrak{G}$ is $\mathfrak{m}$-essentially transitive if there is a conull orbit with respect to $\mathfrak{m}$.   

\begin{prop}\label{essentiallytransitive}
The action of $G$ on $\proba$ is $\mu$-essentially transitive. 
\end{prop}

\begin{proof}
Let $\pi'$ be as in proposition \ref{equivariant}, let $x$ be in $\M_1$ and denote by $\nu$ the image of $x$ by $\pi'$. Let $y \in \M_1$ and choose $g \in G$ such that $gx \in \FM_y$. Such an element exists by proposition \ref{transitive}. Since $\pi'$ is $\mathcal{A}_{sat}$-measurable and thus constant on the leaves of the $\M$-isoperiodic foliation, we have that $g\nu = \pi'(g \cdot x) = \pi'(y)$. This means that the image of $\pi'$ is contained in the $G$-orbit of $\nu$. We compute $\mu(G \cdot \nu) = m\left( \pi^{-1}(G\cdot \nu)\right) $ = $m \left( \pi'^{-1}(G\cdot \nu) \right)$ since $\pi$ and $\pi'$ are equal $m$-almost everywhere and then we deduce $\mu(G\cdot\nu) =1$. We proved that the $G$-orbit of $\nu$ is conull.
\end{proof}

\section{Proof of theorem \ref{main}}

We are now ready to give the proof of our main result. 

\begin{proof}[Proof of theorem A]

Let $\nu \in \proba$ be a point whose $G$-orbit is conull. Such a point is given by proposition \ref{essentiallytransitive} and let $x \in \M_1$ such that $\nu = \pi'(x)$. A classical consequence of Varadarajan's theorem for measurable actions on standard Borel spaces is that the stabilizer of $\nu$ is a closed subgroup $H \subset G$. See corollary 5.8 of \cite{varadarajan1968geometry}. We claim that $H$ contains the stabilizer of the leaf of $x$. Indeed, let $g \in G$ be a matrix in the stabilizer of the leaf of $x$ and since $\pi'$ is $\mathcal{A}_{sat}$-measurable, it follows $g\nu=\pi'(gx) = \pi'(x) = \nu$. Now, let $\Phi: G/H \to \proba$ be the orbit map associated to $\nu$. The space $G/H$ is a standard Borel space and $\Phi$ is an injective $G$-equivariant measurable map. By theorem A.4 in \cite{zimmer2013ergodic}, its corestriction to $\Phi(G/H)$ has a measurable inverse $\Phi^{-1}$ and we denote by $\mu_0$ the pushforward by $\Phi^{-1}$ of the restriction of $\mu$ to $G\cdot \nu$. It is a $G$-invariant probability measure and thus it is the Haar measure. In particular $H$ has finite covolume in $G$. According to Borel density theorem, as stated in Theorem 3.2.5 in \cite{zimmer2013ergodic}, this means that $H$ is either a lattice in $G$ or $G$ itself. 

\medskip 

\noindent In the first case, since $H$ contains the stabilizer of the leaf of $x$, the latter is discrete and thus the space of leaves, which is homeomorphic to the quotient of $G$ by the stabilizer of the leaf of $x$ by proposition \ref{transitive}, is haussdorff. This implies that all the leaves are closed. 

\medskip 

\noindent In the second case, when $H=G$, the $G$-orbit of $\nu$ is a singleton and $\mu$ is the dirac mass at $\nu$. Let $f: \M \to \mathbb{R}$ be a measurable function that is constant on the leaves. By proposition \ref{factors}, there is a measurable map $\mathfrak{f}: \proba \to \mathbb{R}$ such that $f = \mathfrak{f} \circ \pi$ $m$-almost everywhere. This means that $f$ is equal to $\mathfrak{f}(\nu)$ $m$-almost everywhere and thus the foliation is ergodic with respect to $m$. It is a classical fact in ergodic theory of foliations that in that case, $m$-almost every leaf is dense. It follows from proposition \ref{transitive} that all of them are. 
\end{proof}

\bibliographystyle{plain}
\bibliography{bibdeflo.bib}

\begin{thebibliography}{McM14}

\bibitem[CDF15]{calsamiglia2015transfer}
Gabriel Calsamiglia, Bertrand Deroin, and Stefano Francaviglia.
\newblock A transfer principle: from periods to isoperiodic foliations.
\newblock {\em arXiv preprint arXiv:1511.07635}, 2015.

\bibitem[EM18]{eskin2018invariant}
Alex Eskin and Maryam Mirzakhani.
\newblock Invariant and stationary measures for the action on moduli space.
\newblock {\em Publications Math{\'e}matiques de l'IH{\'E}S}, 127(1):95--324,
  2018.

\bibitem[EW13]{einsiedler2013ergodic}
Manfred Einsiedler and Thomas Ward.
\newblock {\em Ergodic theory}.
\newblock Springer, 2013.

\bibitem[Ham18]{hamenstadt2018ergodicity}
Ursula Hamenstädt.
\newblock Ergodicity of the absolute period foliation.
\newblock {\em Israel Journal of Mathematics}, 225:661--680, 2018.

\bibitem[HW15]{hooper2015rel}
W~Patrick Hooper and Barak Weiss.
\newblock Rel leaves of the arnoux--yoccoz surfaces.
\newblock {\em Selecta Mathematica}, pages 1--60, 2015.

\bibitem[McM14]{mcmullen2014moduli}
Curtis McMullen.
\newblock Moduli spaces of isoperiodic forms on riemann surfaces.
\newblock {\em Duke Mathematical Journal}, 163(12):2271--2323, 2014.

\bibitem[Var68]{varadarajan1968geometry}
Veeravalli~Seshadri Varadarajan.
\newblock {\em Geometry of quantum theory}, volume~1.
\newblock Springer, 1968.

\bibitem[Wri14]{wright2014translation}
Alex Wright.
\newblock Translation surfaces and their orbit closures: an introduction for a
  broad audience.
\newblock {\em arXiv preprint arXiv:1411.1827}, 2014.

\bibitem[Ygo20]{ygouf2020criterion}
Florent Ygouf.
\newblock A criterion for density of the isoperiodic leaves in rank 1 affine
  invariant suborbifolds.
\newblock {\em arXiv preprint arXiv:2002.01186}, 2020.

\bibitem[Zim13]{zimmer2013ergodic}
Robert~J Zimmer.
\newblock {\em Ergodic theory and semisimple groups}, volume~81.
\newblock Springer Science \& Business Media, 2013.

\end{thebibliography}

\end{document}